\documentclass[11pt]{article}
\def\smallspace{ \renewcommand{\baselinestretch}{1.4}\small\normalsize }

\usepackage[title]{appendix}
\usepackage{color}
\usepackage{natbib}
\usepackage{tikz}
\usetikzlibrary{patterns}
\usepackage{authblk} 
\usepackage{graphicx}
\usepackage{subfigure} 
\usepackage{tikz}
\usepackage{amssymb}
\usepackage[utf8]{inputenc}
\usepackage{amsthm}
\usepackage{amsmath}
\usepackage{makecell}
\usepackage{array}
\usepackage{multirow}
\usepackage{algorithm}
\floatname{algorithm}{Algorithm}
\usepackage{algorithmic}
\usepackage{comment}
\usepackage{enumitem}  
\usepackage{booktabs}


\newtheorem{lemma}{Lemma}
\newtheorem{proposition}{Proposition}
\newtheorem{theorem}{Theorem}

\newtheorem{remark}{Remark}

\newcommand{\revxzb}[1]{\textcolor{black}{#1}}
\newcommand{\revxzc}[1]{\textcolor{black}{#1}}

\begin{document}
	\pagestyle{plain}
	
	\hbox{\ }\vskip 0pt\thispagestyle{empty}
	\centerline{\LARGE LP Relaxations for Routing and Wavelength Assignment with }\vskip 8pt
	\centerline{\LARGE Partial Path Protection: Formulations and Computations}
	\vskip 12pt
        \centerline{Xianyan Yang\textsuperscript{a}, Junyan Liu\textsuperscript{b}, Fan Zhang\textsuperscript{b}, Fabo Sun\textsuperscript{c}, Feng Li\textsuperscript{d}, Zhou Xu \textsuperscript{a}\footnote{Corresponding author ({zhou.xu@polyu.edu.hk})}}
        \vskip 8pt
        \centerline{\textsuperscript{a} Department of Logistics and Maritime Studies, Hong Kong Polytechnic University, Hong Kong}
        \vskip 2pt
        \centerline{\textsuperscript{b} Theory Lab, 2012 Labs, Huawei Technologies, Co. Ltd}
        \vskip 2pt
        \centerline{\textsuperscript{c} Optical Product Research Department, Huawei Technologies, Co. Ltd}
        \vskip 2pt
        \centerline{\textsuperscript{d} School of Management, Huazhong University of Science and Technology, China}
	
	\vskip 8pt
	\smallspace
	\allowdisplaybreaks
	\centerline{\bf Abstract}
As a variant of the routing and wavelength assignment problem (RWAP), the RWAP with partial path protection (RWAP-PPP) aims to design a reliable optical-fiber network for telecommunications. It needs to assign paths and wavelengths of the network to satisfy communication requests, not only in a normal working situation, but also in every possible failure situation where an optical link of the network has failed. The existing literature lacks any studies on relaxations that can be solved efficiently to produce tight lower bounds on the optimal objective value of the RWAP-PPP. As a result, the quality of solutions to the RWAP-PPP cannot be properly assessed, which is crucial to telecommunication service providers in their bidding to customers and their determination of service improvements. 

\revxzc{Due to the presence of numerous possible failure situations, developing effective lower bounds for the RWAP-PPP is challenging.} To tackle this challenge, we formulate and analyze different linear programming (LP) relaxations of the RWAP-PPP. Among them, we introduce a novel LP relaxation that yields promising lower bounds for the RWAP-PPP. To solve this LP relaxation, we develop a Benders decomposition algorithm, for which a set of valid inequalities is incorporated to enhance the performance. Computational results based on publishable practical networks, including those with hundreds of nodes and edges, 
demonstrate the effectiveness of the newly proposed LP relaxation and the efficiency of its solution algorithm. The obtained lower bounds achieve average optimality gaps of only 8.6\%. \revxzc{Compared with the lower bounds obtained from a direct LP relaxation of the RWAP, which have average optimality gaps of 36.7\%, significant improvements are observed. Consequently, our new LP relaxation and its solution algorithm can be used to effectively assess the solution quality for the RWAP-PPP, holding significant research and practical values.}

	
	\section{Introduction}
	\label{sec:introduction}
	In a Wavelength Division Multiplexing (WDM) optical-fiber telecommunication network, nodes are commonly equipped with network devices such as transmitters, receivers, amplifiers, and switches. They are interconnected by links, which are established using optical fibers. Typically, each link between a pair of connected nodes consists of two optical fibers, allowing for bidirectional communication between them. The WDM technique empowers the transmission of a diverse range of telecommunication signals simultaneously through a single optical fiber by utilizing multiple wavelengths. To fulfill a communication request, \revxzc{a working path that consists of interconnected links spanning from the request's origin to its destination needs to be allocated.} In the absence of wavelength converters, which are very costly, each request's working path is required to occupy the same wavelength on all its links for signal transmission. This is known as \textit{the wavelength continuity constraint}. As a result, \revxzc{each request needs to be assigned a working wavelength for signal transmission throughout all the links of its entire working path.} For any two requests whose allocated working paths share one or more optical links of the network, they cannot be assigned the same working wavelength. This is known as \textit{the wavelength no-clash constraint}.
 
    Given a WDM optical-fiber telecommunication network and a set of communication requests, the \textit{Routing and Wavelength Assignment Problem (RWAP)} involves determining the assignment of both a working path and a working wavelength for each communication request, so that both the wavelength continuity constraint and the wavelength no-clash constraint are satisfied. The RWAP and its variants have been extensively studied due to its wide applications in telecommunication, where various objectives and additional practical constraints have also been taken into account (see, e.g., \citealt{jaumard2006ilp}). However, in many practical situations, solving the RWAP only does not guarantee that the assigned working paths and wavelengths will always fulfill all communication requests. This is primarily because daily operational challenges, such as broken fibers, cause failures in optical links of the network, disrupting the communications.
    
    To protect all communication requests from link failures within the network, it becomes necessary to assign each request backup paths and backup wavelengths, in addition to its working path and working wavelength, forming a  survival optical-fiber network \citep{kennington2003wavelength,zhang2019survivable}.  
    Among various protection schemes, \textit{the full path protection scheme} stands out as the simplest approach. This scheme requires that each communication request must possess edge-disjoint working and backup paths, so that the same backup path can be utilized to fulfill the request in the event of every possible link failure of its working path. In contrast to the full path protection scheme, \textit{the partial path protection scheme} introduces a more complicated approach. Under this scheme, each request is assigned a backup path and a backup wavelength for every possible link failure of its working path. These backup paths and backup wavelengths are allowed to differ for different link failures. A subset of all possible failed links is given, while certain link failures not within this subset are considered either impossible or not of concern. For each possible failed link, the corresponding backup path assigned to a request must exclude that specific failed link, but it does not need to be edge-disjoint from the request's working path. Although more complicated, the partial path protection scheme offers better utilization of the optical links within the network, making it intriguing to practitioners \citep{wang2002partial,Bakri2012}.

    In this paper, we investigate linear programming relaxations for the RWAP with partial path protection (or RWAP-PPP in short). The RWAP-PPP involves determining the assignment of working paths and working wavelengths, along with backup paths and backup wavelengths, for each communication request. Our primary focus is on the problem setting of RWAP-PPP as encountered in prominent telecommunication operators in China, such as China Mobile. In this problem setting, each link of the network is assumed to be undirected. This assumption is based on the requirement that for every communication request, which is defined as a pair of nodes, both the forward and backward signal transmissions between these two nodes are required to be assigned the same paths and wavelengths. This requirement applies to both working and backup purposes, primarily for simplifying solution design and implementation. 
    In addition, the objective value to be minimized in our problem setting of the RWAP-PPP is the total number of wavelengths assigned over all the links within the network, which includes those assigned for the working paths and for the backup paths. This objective value has been studied in the RWAP literature \citep{banerjee2000wavelength,jaumard2006ilp}. It is also served as the crucial performance indicator for those prominent telecommunication operators in China, when evaluating the quality of RWAP-PPP solutions provided by communications technology solution providers like Huawei. Once the communication network is established, a smaller value of this objective indicates a lower occupation of wavelength resources, enabling more available wavelengths for the telecommunication operators to accommodate additional unforeseen demands of communication requests. 
    
    It is worthy to note that the relaxation model and solution method developed in our study can be adapted to other problem settings of the RWAP-PPP. This includes settings considering other objective functions considered in the RWAP, such as to minimize the total number of wavelengths assigned to all communication requests \citep{lee2002optimization,hu2004traffic}, or to minimize the maximum number of wavelengths assigned over all the links \citep{jaumard2006ilp}.

\subsection{Existing Studies on RWAP and RWAP-PPP}
	The RWAP is known to be an NP-hard problem  \citep{chlamtac1992lightpath}. To tackle the RWAP, various integer programming (IP) models have been derived, and both exact and heuristic algorithms have also been developed. \cite{jaumard2006ilp,jaumard2007comparison} propose and compare several IP models of the RWAP, including arc-based, link-based, and path-based models. 
    \cite{jaumard2009column} propose a new IP model based on a maximum-independent-set formulation of the RWAP, and utilize this model to develop a branch-and-bound algorithm, which is used to develop both exact and heuristic algorithms for the RWAP. These IP models are also utilized to develop column-generation-based heuristic algorithms (see, e.g., \citealt{dawande2007traffic,agarwal2016near,jaumard2017efficient}). Computational experiments show that these algorithms can produce near-optimal solutions to instances of 14$\sim$90 nodes, with optimality gaps around 0.1\%$\sim$15.0\%. 
    \revxzc{
    Several other heuristics have also been developed for the RWAP by applications of graph coloring algorithms (\citealt{noronha2006routing}), greedy-based heuristics (\citealt{belgacem2014post}) , and machine learning techniques (\citealt{di2022deep, rai2023analysis}), etc.}
    Additionally, \revxzc{\cite{daryalal2022stochastic} develop a decomposition-based solution approach for a two-stage stochastic RWAP where incorporating requests' uncertainty.}
 
	
	Since the RWAP is a special case of the RWAP-PPP with no possible link failure, the RWAP-PPP is also an NP-hard problem. However, existing studies in the literature focused on only limited special cases of the RWAP-PPP and its variants. For the special case where assignments of wavelengths are not taken into account, \cite{truong2006overlapped} and \cite{truong2007using} develop two heuristic algorithms for a segment-protection variant of the RWAP-PPP, but without examining their optimality gaps in their computational experiments. 
    For such a special case, several existing studies investigate the full path protection variant of the problem, where each communication request needs to be assigned the same backup path for all possible link failures. For example, \cite{jaumard2006backup} derive IP formulations and solve them directly by an optimization solver.  \cite{rocha2008revisiting} and \cite{agarwal2014survivable} develop two heuristics \revxzb{based on} a column-generation solution approach.
    \cite{dey2019offline} develop another heuristic algorithm that leverages the concept of streams in designing backup paths. A stream is defined as a path that combines multiple backup paths of requests, each having disjoint working paths. The stream approach enables rapid communication recovery in the event of link failures while ensuring efficient utilization of wavelength resources. Moreover, when the failure scenario is given, \cite{ dahl1998cutting} propose a cutting-plane algorithm, and \cite{garg2008models} develop a Benders decomposition-based heuristic, to design a survivable network by optimizing the assignment of backup paths. \revxzc{\cite{botton2013benders} and \cite{arslan2020flexible} develop heuristics based on branch-and-cut algorithms to solve survivable network design problems with hop constraints, which restrict the number of links included in each of the working paths and backup paths.}   
    \revxzc{Taking into account the allocation of backup paths and backup wavelengths, \cite{jaumard2012path} present a column-generation-based heuristic. However, their approach is limited to finding nearly optimal solutions for instances with a maximum of 24 nodes only. For the full path protection variant of the problem, \cite{kennington2003wavelength} develop a heuristic based on cycle construction, without examining their optimality gaps.  Additionally, \cite{cseker2023routing} formulate the problem as a quadratic unconstrained binary program and develop a Digital Annealer algorithm based on simulated annealing. }
    

	
Compared with these special cases of various variants studied in the literature and mentioned above, the general case of the RWAP-PPP, considered in this paper, involves significantly more decisions and constraints to optimize, and is much more challenging to solve. Most practical optical networks are very large, with hundreds of nodes, links, and requests. To our knowledge, in practice only heuristic methods have been employed to tackle large-scale instances of the RWAP-PPP. \revxzc{In the literature, several heuristics have been developed for some restricted variant of the RWAP-PPP (see, e.g., \citealt{wang2002partial, xue2007partial,Bakri2012,KOUBAA201456}). However, the quality of the solutions obtained by these heuristics has not been examined, except for two very small  instances of only 14 and 15 nodes considered in \cite{KOUBAA201456}.}

\subsection{Importance of Effective Lower Bounds} 
 \revxzc{While practitioners have developed some heuristics to solve the RWAP-PPP in practical scenarios, they also need effective lower bounds that  can closely approximate the optimal objective value of the problem and can be computed efficiently. Such effective lower bounds are of great importance in both the research and practical applications of the RWAP-PPP. These lower bounds can serve as a benchmark for evaluating the quality of heuristic solutions obtained for the RWAP-PPP.} Due to the lack of such lower bounds, existing studies on heuristics for the RWAP-PPP cannot properly evaluate the gap of their solutions from the optimal objective values. Moreover, in practice, telecommunication solution providers, including Huawei, require such lower bounds to evaluate the potential of improvement on its solution. This is critical in their solution bidding to customers, as well as in their determination of service improvements. They can justify allocating additional resources for improving a specific solution or service plan only when its objective value significantly deviates from a lower bound that is known to be close to the optimal objective value. The availability of such a lower bound can therefore guide their allocation of resources towards areas where substantial improvements can be achieved.

However, effective lower bounds for the RWAP-PPP are currently unknown and have received limited attention in the existing literature. The LP relaxation derived directly from the IP model of the RWAP-PPP is hindered by a significant number of decision variables. As a result, solving it within an affordable running time is not possible, even for some small-scale instances encountered in practice. Additionally, although the optimal objective value of the RWAP provides a valid lower bound, it often deviates significantly from the actual optimal objective value. To address this research gap, our paper focuses on investigating linear programming relaxations for the RWAP-PPP.

 \subsection{Our Contributions} 
 Due to the inherent complexity of problem instances encountered in practice, developing effective lower bounds for the RWAP-PPP is challenging. Such problem instances typically involve numerous failure situations, thousands of requests, and hundreds of nodes and edges in their networks. \revxzc{To tackle this challenge, we derive a novel LP relaxation of the RWAP-PPP and develop an efficient Benders decomposition algorithm to solve the relaxation to exact optimality.}
 
The new relaxation is derived based on the direct LP relaxation of the RWAP-PPP, by relaxing constraints that link decisions on working paths and wavelengths and decisions on backup paths and wavelengths. We show that the model size of such a new relaxation can be further reduced, without changing its optimal objective value, by excluding decisions on working paths and wavelengths, and by aggregating decisions associated with the same wavelength and the same origin of requests. We then prove that our newly proposed LP relaxation can provide a tighter lower bound for the RWAP-PPP than that obtained from the direct LP relaxation of the RWAP, and that the improvement can be arbitrarily large in certain extreme situations.
	
After further analyzing the properties of our newly proposed LP relaxation, we have developed an efficient Benders decomposition algorithm to solve the relaxation. During each iteration of the Benders decomposition algorithm, its first step involves solving a restricted master problem. Subsequently, the optimal solution of the restricted master problem is utilized to expand the constraint set of the restricted master problem. This iteration process is repeated until the optimal solution of the restricted master problem converges to an optimal solution of our newly proposed LP relaxation. \revxzc{To accelerate the Benders decomposition algorithm, the restricted master problem is carefully defined by incorporating a set of valid inequalities.}
 
 We have conducted a computational experiment by using instances on some practical networks. These instances will be made available online for future follow-up studies. The results show that the Benders decomposition algorithm can solve our new LP relaxation to optimality for all the instances efficiently. Our new LP relaxation method is found empirically to provide significantly better lower bounds than those obtained from the direct LP relaxation of the RWAP. We observe improvements of approximately 10.8\% to 52.6\% in the quality of the lower bounds. Additionally, the optimality gaps between our new lower bounds and the objective values of the best-known solutions are only 5.2\% to 14.1\%. These results highlight the effectiveness of our newly proposed LP relaxation and the developed Benders decomposition algorithm in providing accurate estimations of the optimal objective values, enabling us to effectively assess the quality of solutions for the RWAP-PPP.
	
	The remainder of this paper is organized as follows. After presenting the problem definition of the RWAP and RWAP-PPP and their IP formulations in Section~\ref{sec:definition}, 
	we derive some IP relaxations in Section~\ref{sec:IP:relax}. Based on these IP relaxations, we derive and analyze various LP relaxations in Section~\ref{sec:LP:relax}, from which we propose a new LP relaxation that has much fewer decision variables and constraints than a direct LP relaxation of the RWAP-PPP, and that can produce much tighter lower bounds than the direct LP relaxation of the RWAP. In Section~\ref{sec:IP:relax:solution}, we then develop a Benders decomposition algorithm to solve the new LP relaxation. The computational experiment is reported in Section~\ref{sec:experiments}, and the paper is concluded in Section~\ref{sec:conclusions}.
	
\section{Problem Definition and IP Formulation}
\label{sec:definition}
Consider an optical-fiber telecommunication network represented by an undirected graph $G=(V,E)$ with a vertex set $V$ and an edge set $E$. Each vertex $v\in V$ indicates a node $v$ of the telecommunication network, and \revxzc{each edge $e=\{u,v\}\in E$ with $u,v\in V$ indicates an optical link for communication between a pair of nodes $u$ and $v$ of the telecommunication network, which can be from $u$ to $v$ and from $v$ to $u$. (It is possible that multiple edges may exist between a pair of nodes.) Let $K=\{\kappa_1,\kappa_2,\cdots,\kappa_{|K|}\}$ denote the set of wavelengths available on each optical link for their use to facilitate communications on the link. For each $e\in E$, to facilitate communications on link $e$, at least one wavelength in $K$ needs to be assigned to $e$. Each wavelength $\kappa\in K$ can be assigned at most once to the same link.}
	
Let $D$ indicate a set of communication requests, where each request $d\in D$ is represented by an origin-destination pair $(s_d,t_d)$, $s_d\in V$ and $t_d\in V$, indicating a request of communication from node $s_d$ to node $t_d$. To satisfy each request $d\in D$, one needs to assign $d$ a working path $P^0_d$ from $s_d$, through some links of $G$, to $t_d$. One also needs to assign all links of the working path $P^0_d$ the same working wavelength in $K$, which is denoted by $\lambda^0_d\in K$. This ensures the satisfaction of a \textit{wavelength continuity constraint}. Moreover, to ensure that every wavelength of a link is used at most once in satisfying the communication requests by working paths, the following \textit{no-clash constraint} also needs to be satisfied: 
     For any two requests $d$ and $d'$, if their assigned working wavelengths $\lambda^0_d$ and $\lambda^0_{d'}$ are the same, then their assigned working paths $P^0_{d}$ and $P^0_{d'}$ cannot share any common links. 
	The routing and wavelength assignment problem (or RWAP in short) is to find such assignments of working paths and working wavelengths that minimize the total number of assigned wavelengths over all the links.
	
	The RWAP-PPP is an extension of the RWAP, and it aims to ensure that all communication requests can still be satisfied in every possible failure situation where an optical link of the telecommunication network has failed. In particular, let $\Pi \subseteq E$ indicate a subset of links that can possibly fail. For each link $\tau\in \Pi$, we need to determine a backup path $P^\tau_d$ as well as assign all links of $P^\tau_d$ the same backup wavelength $\lambda^\tau_d\in K$, so as to satisfy each request $d\in D$. Consider the following two cases of request $d$:
	\begin{itemize}
	    \item If its working path $P^0_d$ does not contain $\tau$, request $d$ can still be satisfied by $P^0_d$ and its working wavelength $\lambda^0_d$, so that we require $P^\tau_d=P^0_d$ and $\lambda^\tau_d=\lambda^0_d$.
	    \item If its working path $P^0_d$ contains $\tau$, one needs to assign $d$ a different path that does not contain $\tau$, and so it is required that $\tau$ cannot appear in $P^\tau_d$.
	\end{itemize}
	All the backup paths $P^{\tau}_d$ for $d\in D$ also need to satisfy the following \textit{no-clash constraint}, 
	\begin{itemize}
	    \item For any two requests $d$ and $d'$, if their assigned backup wavelengths $\lambda^\tau_d$ and $\lambda^\tau_{d'}$ are the same, their assigned backup paths $P^\tau_{d}$ and $P^\tau_{d'}$ cannot share any common links. 
	\end{itemize}
	
The RWAP-PPP is to find such assignments of working paths and working wavelengths as well as backup paths and backup wavelengths that minimize the total number of assigned wavelengths over all the links. See an illustrative example of the problem instance in Appendix~\ref{sec:app:eg:rwappp}.

The RWAP-PPP can be formulated as an IP model. To establish such an IP model, we need to first introduce the following parameters:
\begin{itemize}
    \item For each (undirected) link $e=\{u,v\}\in E$, let $\alpha_1(e)=(u,v)$ indicate a directed arc from node $u$ to node $v$ \revxzc{associated with edge $e$}, and $\alpha_2(e)=(v,u)$ indicate a directed arc from $v$ to $u$ \revxzc{associated with edge $e$}. 
    \item Let $A=\{\alpha_1(e):e\in E\}\cup \{\alpha_2(e):e\in E\}$ denote the set of all directed arcs derived from links in $E$. 
    \item For each node $v\in V$, let $A^+(v)\subseteq A$ indicate the set of outgoing arcs from $v$, and $A^-(v)\subseteq A$ indicate the set of incoming arcs to $v$.
\end{itemize} 
Next, we define the following decision variables:
\begin{itemize}
    \item For each $d\in D$, $a\in A$, and $\kappa\in K$, let $x^{\kappa,a}_{d}\in \{0,1\}$ indicate a binary variable, which equals $1$ if and only if the working path of request $d$ from $s_d$ to $t_d$ goes through arc $a$, and $\kappa$ is assigned as the working wavelength.
    \item For each $\tau\in \Pi$, $d\in D$, $a\in A$, and $\kappa\in K$, let $y^{\kappa,a}_{\tau,d}\in \{0,1\}$ indicate a binary variable, which equals $1$ if and only if in the situation when link $\tau$ has failed, the backup path of request $d$ from $s_d$ to $t_d$ goes through arc $a$, and $\kappa$ is assigned as the backup wavelength.    
    \item For each $e\in E$ and $\kappa\in K$, let $w^{\kappa,e}\in \{0,1\}$ indicate a binary variable, which equals $1$ if and only if wavelength $\kappa$ is assigned to the communication of link $e$.
\end{itemize}

Accordingly, the RWAP-PPP can be formulated into the following IP model:
\begin{align}
\mbox{($\mathrm{IP_{RWAP-PPP}}$)} \quad & \min \sum_{\kappa\in K}\sum_{e\in E}w^{\kappa,e} \label{eqn:ip:obj}\\
\mbox{s.t.}~ 
& \sum_{\kappa\in K}\sum_{a \in A^+(s_d)}x^{\kappa,a}_{d} = 1, \mbox{ $\forall d\in D$,} \label{eqn:ip:x:s}\\ 
& \sum_{\kappa\in K}\sum_{a \in A^-(s_d)}x^{\kappa,a}_{d} = 0, \mbox{ $\forall d\in D$,} \label{eqn:ip:x:s-}\\ 
& \sum_{a \in A^-(v)}x^{\kappa,a}_{d} = \sum_{a \in A^+(v)}x^{\kappa,a}_{d}, \mbox{ $\forall d\in D, ~\kappa\in K,~v\in V\setminus\{s_d,t_d\}$}, \label{eqn:ip:x:v}\\
& \sum_{d\in D}(x^{\kappa,\alpha_1(e)}_{d}+x^{\kappa,\alpha_2(e)}_{d})  \leq w^{\kappa,e}, \mbox{ $\forall \kappa\in K,~e\in E$}, \label{eqn:ip:def:w1}\\
& \sum_{a \in A^+(s_d)}\sum_{\kappa\in K}y^{\kappa,a}_{\tau,d} = 1, \mbox{ $\forall \tau\in \Pi,~\forall d\in D$}, \label{eqn:ip:y:s}\\
& \sum_{a \in A^-(s_d)}\sum_{\kappa\in K}y^{\kappa,a}_{\tau,d} = 0, \mbox{ $\forall \tau\in \Pi,~\forall d\in D$}, \label{eqn:ip:y:s-}\\
& \sum_{a \in A^-(v)}y^{\kappa,a}_{\tau,d} = \sum_{a \in A^+(v)}y^{\kappa,a}_{\tau,d}, \mbox{ $\forall \tau\in \Pi,~d\in D,~\kappa\in K,~v\in V\setminus\{s_d,t_d\}$}, \label{eqn:ip:y:v}\\
&  \sum_{d\in D}(y^{\kappa,\alpha_1(e)}_{\tau,d} + y^{\kappa,\alpha_2(e)}_{\tau,d}) \leq w^{\kappa,e}, \mbox{ $\forall \tau\in \Pi,~\kappa\in K,~e\in E$}, \label{eqn:ip:def:w2}\\
& x^{\kappa,a}_{d} - \sum_{\kappa'\in K}(x^{\kappa',\alpha_1(\tau)}_d+x^{\kappa',\alpha_2(\tau)}_d)\leq y^{\kappa,a}_{\tau,d}, \mbox{ $\forall \tau\in \Pi,~d\in D,~\kappa\in K,~a\in A$},\label{eqn:ip:xy1}\\
& y^{\kappa,a}_{\tau,d}\leq x^{\kappa,a}_{d} + \sum_{\kappa'\in K}(x^{\kappa',\alpha_1(\tau)}_d+x^{\kappa',\alpha_2(\tau)}_d), \mbox{ $\forall \tau\in \Pi,~d\in D,~\kappa\in K,~a\in A$},\label{eqn:ip:xy2}\\
& y^{\kappa,\alpha_1(\tau)}_{\tau,d}+y^{\kappa,\alpha_2(\tau)}_{\tau,d}=0, \mbox{ $\forall \tau\in \Pi,~d\in D,~\kappa\in K$},\label{eqn:ip:yy}\\
& x^{\kappa,a}_{d}\in \{0,1\},\mbox{ $\forall
d\in D,~\kappa\in K,a\in A$}, \label{eqn:ip:x}\\
& w^{\kappa,e}\in \{0,1\}, \mbox{ $\forall
\kappa\in K,~e\in E$}. \label{eqn:ip:w}\\
& y^{\kappa,a}_{\tau,d}\in \{0,1\}, \mbox{ $\forall
\tau\in \Pi,~d\in D,~\kappa\in K,a\in A$}, \label{eqn:ip:y}
\end{align}
In model $\mathrm{IP_{RWAP-PPP}}$ above, the objective (\ref{eqn:ip:obj}) is to minimize the total number of wavelengths assigned to the links of the network. Constraints (\ref{eqn:ip:x:s}) ensure that each request $d$ has exactly one arc leaving from its origin $s_d$ in its working path, and it is assigned exactly one working wavelength on this arc.  Constraints (\ref{eqn:ip:x:s-}) ensure that each request $d$ does not have any arcs towards its origin $s_d$ in its working path, and it is not assigned any working wavelength on such arcs. The balancing constraints (\ref{eqn:ip:x:v}), in conjunction with (\ref{eqn:ip:x:s}), guarantee that each request $d$ is assigned a working path where all links are assigned the same working wavelength. For each request $d$, it is important to note that only its origin $s_d$ and destination $t_d$ are not involved in the balancing constraints (\ref{eqn:ip:x:v}). Accordingly, from (\ref{eqn:ip:x:s})--(\ref{eqn:ip:x:v}), we can see that the working path assigned to request $d$ must begin at $s_d$ and end at $t_d$. 

Moreover, constraints (\ref{eqn:ip:y:s})--(\ref{eqn:ip:y:v}) ensure that in the situation where link $\tau\in \Pi$ has failed, each $d\in D$ is assigned a backup path in which all links are assigned the same backup wavelength. Constraints~(\ref{eqn:ip:def:w1}) and (\ref{eqn:ip:def:w2}) ensure that the no-clash constraints are satisfied by working paths, backup paths, and their assigned wavelengths. Constraints~(\ref{eqn:ip:xy1}) and (\ref{eqn:ip:xy2}) become redundant if $\sum_{\kappa'\in K}(x^{\kappa',\alpha_1(\tau)}_d+x^{\kappa',\alpha_2(\tau)}_d)=1$. However, if $\sum_{\kappa'\in K}(x^{\kappa',\alpha_1(\tau)}_d+x^{\kappa',\alpha_2(\tau)}_d)=0$, indicating that link $\tau$ is not contained in the working path of request $d$, then constraints~(\ref{eqn:ip:xy1}) and (\ref{eqn:ip:xy2})
imply that $y^{\kappa,a}_{\tau,d}=x^{\kappa,a}_{d}$. This ensure that in the situation where link $\tau$ has failed, for each request $d$, if its working path does not contain $\tau$, then its backup path and backup wavelength are the same as its working path and working wavelength, respectively. Additionally, constraints~(\ref{eqn:ip:yy}) ensure that in the situation where link $\tau$ has failed, no backup path can contain link $\tau$. Constraints (\ref{eqn:ip:x})--(\ref{eqn:ip:w}) are binary constraints of the decision variables. Since $|\Pi|$ and $|A|$ are both in $O(|E|)$, model $\mathrm{IP_{RWAP-PPP}}$ above contains $O(|D||K||E|^2)$ binary variables and $O(|D||K||E|^2)$ constraints, making it very challenging to solve directly.

\section{IP Relaxations}
\label{sec:IP:relax}
In this section, we derive several relaxations of the RWAP-PPP, which can be formulated as IP models. First, by keeping only variables $x$ and $w$, as well as constraints on them, we can obtain from model $\mathrm{IP_{RWAP-PPP}}$ a relaxation, which is the IP formulation of the RWAP, and is thus referred to as model $\mathrm{IP_{RWAP}}$.
\begin{align*}
\mbox{($\mathrm{IP_{RWAP}}$)}  \quad & \min \sum_{\kappa\in K}\sum_{e\in E}w^{\kappa,e} \\
\mbox{s.t.}~ 
& \sum_{\kappa\in K}\sum_{a \in A^+(s_d)}x^{\kappa,a}_{d} = 1, \mbox{ $\forall d\in D$,} \\ 
& \sum_{\kappa\in K}\sum_{a \in A^-(s_d)}x^{\kappa,a}_{d} = 0, \mbox{ $\forall d\in D$,} \\
& \sum_{a \in A^-(v)}x^{\kappa,a}_{d} = \sum_{a \in A^+(v)}x^{\kappa,a}_{d}, \mbox{ $\forall d\in D, ~\kappa\in K,~v\in V\setminus\{s_d,t_d\}$}, \\
& \sum_{d\in D}(x^{\kappa,\alpha_1(e)}_{d}+x^{\kappa,\alpha_2(e)}_{d})  \leq w^{\kappa,e}, \mbox{ $\forall \kappa\in K,~e\in E$}, \\
& x^{\kappa,a}_{d}\in \{0,1\},\mbox{ $\forall
d\in D,~\kappa\in K,a\in A$}, \\
& w^{\kappa,e}\in \{0,1\}, \mbox{ $\forall
\kappa\in K,~e\in E$}.
\end{align*}
Model $\mathrm{IP_{RWAP}}$ has only $O(|D||K||E|)$ binary variables and $O(|D||K||V|+|K||E|)$ constraints, which are much fewer than those of model $\mathrm{IP_{RWAP-PPP}}$.

Next, notice that in model $\mathrm{IP_{RWAP-PPP}}$, there are $O(|E|^2|D||K|)$ constraints in (\ref{eqn:ip:xy1}) and (\ref{eqn:ip:xy2}). By relaxing these constraints, we obtain another IP relaxation of model $\mathrm{IP_{RWAP-PPP}}$:
\begin{align*}
\mbox{($\mathrm{IP_{R1}}$)}  \quad & \min~(\ref{eqn:ip:obj})\\
\mbox{s.t.}~ 
& \mbox{(\ref{eqn:ip:x:s})--(\ref{eqn:ip:def:w2}),~(\ref{eqn:ip:yy})--(\ref{eqn:ip:y}).}
\end{align*}
By further excluding variables $x$ and their related constraints, we obtain another IP relaxation of model $\mathrm{IP_{RWAP-PPP}}$:
\begin{align*}
\mbox{($\mathrm{IP_{R2}}$)}  \quad & \min~(\ref{eqn:ip:obj})\\
\mbox{s.t.}~ 
& \mbox{(\ref{eqn:ip:y:s})--(\ref{eqn:ip:def:w2}),~(\ref{eqn:ip:yy}),~(\ref{eqn:ip:w}),~(\ref{eqn:ip:y}).}
\end{align*}
Both models $\mathrm{IP_{R1}}$ and $\mathrm{IP_{R2}}$ contains $O(|D||K||E|^2)$ variables and $O(|D||K||E||V|+|K||E|^2)$ constraints, much fewer than $\mathrm{IP_{RWAP-PPP}}$. Since $\mathrm{IP_{R2}}$ does not contain variables $x$ and their related constraints, it has $O(|D||K||E|)$ and $O(|D||K||V|+|K||E|)$ fewer variables and constraints, respectively, than $\mathrm{IP_{R1}}$.

Let $z_{\mathrm{IP_{R1}}}$ and $z_{\mathrm{IP_{R2}}}$ denote the optimal objective values of models $\mathrm{IP_{R1}}$ and $\mathrm{IP_{R2}}$, respectively. Lemma~\ref{lemma:ipr12} shows that the two IP relaxations above, $\mathrm{IP_{R1}}$ and $\mathrm{IP_{R2}}$, are equivalent.
\begin{lemma}
\label{lemma:ipr12}
$z_{\mathrm{IP_{R1}}}=z_{\mathrm{IP_{R2}}}$.
\end{lemma}
\begin{proof}{Proof}
Since model $\mathrm{IP_{R2}}$ is a relaxation of model $\mathrm{IP_{R1}}$, we obtain that $z_{\mathrm{IP_{R1}}}\geq z_{\mathrm{IP_{R2}}}$. Let $\hat{y}$ and $\hat{w}$ indicate the optimal solution to model $\mathrm{IP_{R2}}$. Let $\tau'$ denote any link in $\Pi$. Define $\hat{x}_{d}^{\kappa,a}=\hat{y}_{\tau',d}^{\kappa,a}$ for $d\in D$, $\kappa\in K$, and $a\in A$. Since $\hat{y}$ and $\hat{w}$ satisfy (\ref{eqn:ip:y:s})--(\ref{eqn:ip:def:w2}) for $\tau=\tau'$, we obtain that $\hat{x}$ and $\hat{w}$ satisfy (\ref{eqn:ip:x:s})--(\ref{eqn:ip:def:w1}). This implies that $\hat{x}$, $\hat{y}$, and $\hat{w}$ form a feasible solution to model $\mathrm{IP_{R1}}$. Thus, we obtain that $z_{\mathrm{IP_{R1}}}\leq z_{\mathrm{IP_{R2}}}$. Hence, $z_{\mathrm{IP_{R1}}}=z_{\mathrm{IP_{R2}}}$, and Lemma~\ref{lemma:ipr12} is proved.
\end{proof}

Let $z_{\mathrm{IP_{RWAP-PPP}}}$ and $z_{\mathrm{IP_{RWAP}}}$ indicate the optimal objective values of models $\mathrm{IP_{RWAP-PPP}}$ and $\mathrm{IP_{RWAP}}$, respectively. From Lemma~\ref{lemma:ipr12}, we can directly obtain Theorem~\ref{thm:IP} below. 
\begin{theorem}
\label{thm:IP}
    $z_{\mathrm{IP_{RWAP-PPP}}} \geq z_{\mathrm{IP_{R1}}}=z_{\mathrm{IP_{R2}}}\geq z_{\mathrm{IP_{RWAP}}}$.
\end{theorem}

Moreover, Proposition~\ref{prop:ZIPR2} below indicates that the lower bound provided by $\mathrm{IP_{RWAP}}$ (without partial path protections taken into account) can be arbitrarily worse than that provided by $\mathrm{IP_{R2}}$ (with partial path protections taken into account) in certain extreme situations.
\begin{proposition}
\label{prop:ZIPR2}
Given any constant $L>0$, there always exists an instance of the RWAP-PPP such that the ratio $z_{\mathrm{IP_{R2}}}/z_{\mathrm{IP_{RWAP}}}$ exceeds $L$.
\end{proposition}
\begin{proof}{Proof}
\begin{figure}[t]
    \centering
    \caption{An instance on a cycle of $m$ nodes for the proof of Proposition~\ref{prop:ZIPR2}.}
    \vspace{6pt}
    \scalebox{0.8}{	\includegraphics{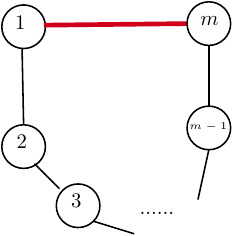}}
    \label{fig:prop1:eg}
\end{figure}
Consider an instance of the RWAP-PPP shown in Figure~\ref{fig:prop1:eg}, where the underlying network is a cycle consisting of $m\geq 3$ nodes, and where $|K|\geq n$. The $n$ communication requests are all from origin node 1 to destination node $m$. All optical links are possible to fail.

For the RWAP, it is optimal to satisfy these communication requests via link $\{1,m\}$ using $n$ wavelengths, implying that the optimal objective value of model $\mathrm{IP_{RWAP}}$ equals $n$, i.e., $z_{\mathrm{IP_{RWAP}}}=n$. 

For model $\mathrm{IP_{R2}}$, it can be simplified as follows:
\begin{align}
    \min \quad & y_{1m} + (m-1)y'_{1m}\\
    \mbox{s.t.}~
    & y'_{1m} = n, \label{eqn:ipr2:simple:1}\\
    & y_{1m}  = n, \label{eqn:ipr2:simple:2}\\
    & y_{1m}, y'_{1m} \in \mathbb{Z}^{+},
\end{align}
where $y'_{1m}$ indicates the number of requests satisfied by path $(1,2,\cdots,m-1,m)$ when link $\{1,m\}$ has failed, and $y_{1m}$ indicates the number of requests satisfied by paths $(1,m)$ when any one of the links in $\{\{1,2\},\{2,3\},\cdots,\{m-1,m\}\}$ has failed. When link $\{1,m\}$ has failed, each request can be satisfied only by path $(1,2,\cdots,m-1,m)$. When any one of the links in $\{\{1,2\},\{2,3\},\cdots,\{m-1,m\}\}$ has failed, each request can be satisfied only by path $(1,m)$. Thus, $y_{1m}  = y'_{1m} = n$ must be satisfied, leading to constraints (\ref{eqn:ipr2:simple:1}) and (\ref{eqn:ipr2:simple:2}). Therefore, the objective value of the above model equals a constant $n+(m-1)n=mn$. This implies that $z_{\mathrm{LP_{R2}}}=mn$.

Therefore, for any given constant $L>0$, by setting $m=\lceil L\rceil +2\geq 3$ in the instance above, we can ensure that $z_{\mathrm{IP_{R2}}}=mn=(\lceil L\rceil +2)n$ and $z_{\mathrm{IP_{RWAP}}}=n$, so that the ratio $z_{\mathrm{IP_{R2}}}/z_{\mathrm{IP_{RWAP}}}=\lceil L\rceil +2$, which exceeds $L$. This completes the proof of Proposition~\ref{prop:ZIPR2}.
\end{proof}

\section{LP Relaxations}
\label{sec:LP:relax}
Let $\mathrm{LP_{RWAP-PPP}}$, $\mathrm{LP_{R1}}$, $\mathrm{LP_{R2}}$, and $\mathrm{LP_{RWAP}}$ indicate the LP relaxations of models $\mathrm{IP_{RWAP-PPP}}$, $\mathrm{IP_{R1}}$, and $\mathrm{IP_{R2}}$, and $\mathrm{IP_{RWAP}}$, respectively. The optimal objective values of these four LP relaxations, denoted by $z_{\mathrm{LP_{RWAP-PPP}}}$, $z_{\mathrm{LP_{R1}}}$, $z_{\mathrm{LP_{R2}}}$, and $z_{\mathrm{LP_{RWAP}}}$, respectively, can all provide lower bounds on the optimal objective value of model $\mathrm{IP_{RWAP-PPP}}$. 

From the IP formulations in Section~\ref{sec:IP:relax}, we know that the direct LP relaxation $\mathrm{LP_{RWAP-PPP}}$ of the RWAP-PPP contains $O(|D||K||E|^2)$ variables and $O(|D||K||E|^2)$ constraints, making it very challenging to solve directly. Moreover, both LP relaxations $\mathrm{LP_{R1}}$ and $\mathrm{LP_{R2}}$ contain $O(|D||K||E|^2)$ variables and $O(|D||K||E||V|+|K||E|^2)$ constraints, much fewer than $\mathrm{LP_{RWAP-PPP}}$. However, LP relaxation $\mathrm{LP_{R2}}$ does not contain variables $x$ and their related constraints, and thus it has $O(|D||K||E|)$ and $O(|D||K||V|+|K||E|)$ fewer variables and constraints, respectively, than $\mathrm{LP_{R1}}$.

By following an analysis similar to that for the establishment of Theorem~\ref{thm:IP}, we can also establish Theorem~\ref{thm:LP} below, implying that $\mathrm{LP_{R2}}$ is equivalent to $\mathrm{LP_{R1}}$ and tighter than $\mathrm{LP_{RWAP}}$.  
\begin{theorem}
\label{thm:LP}
    $z_{\mathrm{IP_{RWAP-PPP}}}\geq z_{\mathrm{LP_{RWAP-PPP}}} \geq z_{\mathrm{LP_{R1}}}=z_{\mathrm{LP_{R2}}}\geq z_{\mathrm{LP_{RWAP}}}$.
\end{theorem}

Proposition~\ref{proposition:zlpr2} below indicates that the lower bound provided by $\mathrm{LP_{RWAP}}$ (without partial path protections taken into account) can be arbitrarily worse than that provided by $\mathrm{LP_{R2}}$ (with partial path protections taken into account) in certain extreme situations.
\begin{proposition}
\label{proposition:zlpr2}
Given any constant $L>0$, there always exists an instance of the RWAP-PPP such that the ratio $z_{\mathrm{LP_{R2}}}/z_{\mathrm{LP_{RWAP}}}$ exceeds $L$.
\end{proposition}
\begin{proof}{Proof}
Let us still consider the instance of the RWAP-PPP shown in Figure~\ref{fig:prop1:eg} and used in the proof of Proposition~\ref{prop:ZIPR2}, where all links are possible to fail. For this instance, we can show that the optimal objectives $z_{\mathrm{LP_{R2}}}$ and $z_{\mathrm{LP_{RWAP}}}$ of the LP relaxations of models $\mathrm{IP_{R2}}$ and $\mathrm{IP_{RWAP}}$ are equal to the optimal objective values of models $\mathrm{IP_{R2}}$ and $\mathrm{IP_{RWAP}}$, respectively. More specifically, we can show that $z_{\mathrm{LP_{R2}}}=z_{\mathrm{IP_{R2}}}=mn$ and $z_{\mathrm{LP_{RWAP}}}=z_{\mathrm{IP_{RWAP}}}=n$.

First, to show $z_{\mathrm{LP_{RWAP}}}=z_{\mathrm{IP_{RWAP}}}=n$, we simplify model $\mathrm{LP_{RWAP}}$, the LP relaxation of model $\mathrm{IP_{RWAP}}$, as follows:
\begin{align*}
    \min \quad & x_{1m} + (m-1) x'_{1m}\\
    \mbox{s.t.}~ & x_{1m} + x'_{1m} = n,\\
    & x_{1m}, x'_{1m} \in \mathbb{R}^{+},
\end{align*}
where $x_{1m}$ and $x'_{1m}$ indicate the numbers of requests satisfied by paths $(1,m)$ and $(1,2,\cdots,m-1,m)$, respectively. Since $m\geq 3$, which implies that $m-1>1$, the optimal objective value of this LP relaxation must equal $n$, i.e., $z_{\mathrm{LP_{RWAP}}}=n$ (with $x_{1m} = n$ and $x'_{1m}=0$). As explained in the proof of Proposition~\ref{prop:ZIPR2}, we have $z_{\mathrm{IP_{RWAP}}}=n$, implying that $z_{\mathrm{LP_{RWAP}}}=z_{\mathrm{IP_{RWAP}}}=n$.

Next, to show $z_{\mathrm{LP_{R2}}}=z_{\mathrm{IP_{R2}}}=mn$, we consider a simplified model of model $\mathrm{LP_{R2}}$ as follows, which can be derived directly from the simplified model of $\mathrm{IP_{R2}}$ proposed in the proof of Proposition~\ref{prop:ZIPR2}:
\begin{align*}
    \min \quad  & y_{1m} + (m-1)y'_{1m}\\
    \mbox{s.t.}~ 
    & y'_{1m} = n,\\
    & y_{1m}  = n,\\
    & y_{1m}, y'_{1m} \in \mathbb{R}^{+}.
\end{align*}
We can see that the objective value of this LP relaxation equals constant $mn$, which is the same as the objective value of model $\mathrm{IP_{R2}}$ shown in the proof of Proposition~\ref{prop:ZIPR2}. Thus, $z_{\mathrm{LP_{R2}}}=z_{\mathrm{IP_{R2}}}=mn$.

Therefore, similar to the proof of Proposition~\ref{prop:ZIPR2}, for any given constant $L>0$, by setting $m=\lceil L\rceil +2\geq 3$ in the instance above, we can ensure that $z_{\mathrm{LP_{R2}}}=mn=(\lceil L\rceil +2)n$ and $z_{\mathrm{LP_{RWAP}}}=n$, so that the ratio $z_{\mathrm{LP_{R2}}}/z_{\mathrm{LP_{RWAP}}}=\lceil L\rceil +2$, which exceeds $L$. This completes the proof of Proposition~\ref{proposition:zlpr2}.
\end{proof}

Next, let $q_{st}$ indicate the number of requests with their origin at node $s$ and destination at node~$t$ for $s\in V$ and $t\in V$. Consider a new LP relaxation $\mathrm{LP_{R3}}$ of the RWAP-PPP shown below, which relaxes model $\mathrm{LP_{R2}}$ with variables and constraints aggregated by wavelengths and by the origins of communication requests:
\begin{align}
\mbox{($\mathrm{LP_{R3}}$)} \quad & \min \sum_{e\in E}\bar{w}^{e} \label{eqn:lpr3:obj}\\
\mbox{s.t.}~ 
& \sum_{a \in A^+(s)}y^{a}_{\tau,s} = \sum_{t\in T}q_{st}, \mbox{ $\forall \tau\in \Pi,~\forall s\in V$}, \label{eqn:lpr3:y:s}\\
& \sum_{a \in A^-(s)}y^{a}_{\tau,s} = 0, \mbox{ $\forall \tau\in \Pi,~\forall s\in V$}, \label{eqn:lpr3:y:s-}\\
& \sum_{a \in A^-(v)}y^{a}_{\tau,s} = \sum_{a \in A^+(v)}y^{a}_{\tau,s}+q_{sv}, \mbox{ $\forall \tau\in \Pi,~s\in V,~v\in V\setminus\{s\}$}, \label{eqn:lpr3:y:v}\\
&  \sum_{s\in V}(y^{\alpha_1(e)}_{\tau,s} + y^{\alpha_2(e)}_{\tau,s}) \leq \bar{w}^{e}, \mbox{ $\forall \tau\in \Pi,~e\in E$}, \label{eqn:lpr3:def:w2}\\
& y^{\alpha_1(\tau)}_{\tau,s}+y^{\alpha_2(\tau)}_{\tau,s}=0, \mbox{ $\forall \tau\in \Pi,~s\in V$},\label{eqn:lpr3:yy}\\
& 0\leq \bar{w}^{e}\leq |K|, \mbox{ $\forall
e\in E$}, \label{eqn:lpr3:w}\\
& 0\leq y^{a}_{\tau,s}\leq \sum_{t\in T}q_{st}, \mbox{ $\forall
\tau\in \Pi,~s\in V,~a\in A$}. \label{eqn:lpr3:y}
\end{align}
In $\mathrm{LP_{R3}}$, each variable $\bar{w}^{e}$ for $e\in E$ indicates the sum of variables $w^{\kappa,e}$ of $\mathrm{LP_{R2}}$ for all $\kappa\in K$, and each variable $y^{a}_{\tau,s}$ for $\tau\in \Pi$, $s\in V$, and $a\in A$ indicates the sum of variables $y^{\kappa,a}_{\tau,d}$ for all $\kappa\in K$ and $d\in D$ with $s_d=s$. Since $|E|$ and $|A|$ are both in $O(|E|)$, it can be seen that 
$\mathrm{LP_{R3}}$ contains only $O(|E|^2|V|)$ variables and $O(|E||V|^2)$ constraints, much fewer than $\mathrm{LP_{R2}}$, which, as explained above, has $O(|D||K||E|^2)$ variables and $O(|D||K||E||V|+|K||E|)$ constraints.

Moreover, let $z_{\mathrm{LP_{R3}}}$ indicate the optimal objective value of $\mathrm{LP_{R3}}$. Proposition~\ref{prop:R3-R2} below shows that although the number of variables and constraints are significantly reduced, the new LP relaxation $\mathrm{LP_{R3}}$ is equivalent to $\mathrm{LP_{R2}}$.
\begin{proposition}
\label{prop:R3-R2}
$z_{\mathrm{LP_{R3}}}=z_{\mathrm{LP_{R2}}}$.
\end{proposition}
\begin{proof}{Proof}
On the one hand, due to the aggregation of variables and constraints, $\mathrm{LP_{R3}}$ is a relaxation of $\mathrm{LP_{R2}}$, which implies that $z_{\mathrm{LP_{R3}}}\leq z_{\mathrm{LP_{R2}}}$. 

On the other hand, let $\hat{y}$ and $\hat{w}$ indicate the optimal solution to $\mathrm{LP_{R3}}$. For each $\tau\in \Pi$ and $s\in V$, we can see from (\ref{eqn:lpr3:y:s})--(\ref{eqn:lpr3:y:v}) that values of variables $\hat{y}_{\tau,s}^{a}$ for $a\in A$ form a network flow that originates from $s$ and that has a quantity of $q_{s,t}$ being absorbed at each node $t\in V\setminus \{s\}$. Thus, for each $s\in V$, there must exist non-negative values for new variable $\bar{y}^{a}_{\tau,s,t}$ for $a\in A$ and $t\in V\setminus \{s\}$, such that
\begin{align}
&\sum_{t\in V\setminus \{s\}}\bar{y}^{a}_{\tau,s,t} = \hat{y}^{a}_{\tau,s}, \label{proof:R2-R2:sum}\\
&\sum_{a\in A^-(v)}\bar{y}^{a}_{\tau,s,t} = \sum_{a\in A^+(v)}\bar{y}^{a}_{\tau,s,t}, \mbox{ for $v\in V\setminus \{s,t\}$}, \label{proof:R3-R2:balance}\\
&\sum_{a\in A^-(t)}\bar{y}^{a}_{\tau,s,t} = q_{s,t}, \label{proof:R2-R2:qst}\\
&\sum_{a\in A^+(t)}\bar{y}^{a}_{\tau,s,t} = 0. \label{proof:R2-R2:0}
\end{align}
Here, each new variable $\bar{y}^{a}_{\tau,s,t}$ represents the flow that originates from node $s$, passes arc $a$, and is absorbed at node $t$. From these above, we can obtain that for $t\in V\setminus \{s\}$, 
\begin{align}
&\sum_{a\in A^+(s)}\bar{y}^{a}_{\tau,s,t} = q_{s,t}, \label{proof:R3-R2:source}\\
&\sum_{a\in A^-(s)}\bar{y}^{a}_{\tau,s,t} = 0. \label{proof:R3-R2:source-}
\end{align}

We next define 
\begin{align*}
\tilde{y}^{\kappa,a}_{\tau,d} = \frac{\bar{y}^{a}_{\tau,s_d,t_d}}{|K|q_{s_d,t_d}} \mbox{ for $\tau\in \Pi,d\in D,\kappa\in K,a\in A$, and } \tilde{w}^{\kappa,e} = \frac{\hat{w}^{e}}{|K|} \mbox{ for $\kappa\in K, e\in E$}.
\end{align*}
We can now verify as follows that $\tilde{y}$ and $\tilde{w}$ defined above satisfy constraints (\ref{eqn:ip:y:s})--(\ref{eqn:ip:def:w2}), and (\ref{eqn:ip:yy}). For each $\tau\in \Pi$, $d\in D$, by (\ref{proof:R3-R2:source}) and (\ref{proof:R3-R2:source-}) we have that
\begin{align*}
\sum_{a \in A^+(s_d)}\sum_{\kappa\in K}\tilde{y}^{\kappa,a}_{\tau,d} = \sum_{a \in A^+(s_d)}\sum_{\kappa\in K} \frac{\bar{y}^{a}_{\tau,s_d,t_d}}{|K|q_{s_d,t_d}} = \frac{\sum_{a \in A^+(s_d)}\bar{y}^{a}_{\tau,s_d,t_d}}{q_{s_d,t_d}} = 1,\\
\sum_{a \in A^-(s_d)}\sum_{\kappa\in K}\tilde{y}^{\kappa,a}_{\tau,d} = \sum_{a \in A^-(s_d)}\sum_{\kappa\in K} \frac{\bar{y}^{a}_{\tau,s_d,t_d}}{|K|q_{s_d,t_d}} = \frac{\sum_{a \in A^-(s_d)}\bar{y}^{a}_{\tau,s_d,t_d}}{q_{s_d,t_d}} = 0,
\end{align*}
implying that (\ref{eqn:ip:y:s}) and (\ref{eqn:ip:y:s-}) are satisfied. 

For each $\tau\in \Pi$, $d\in D$, $\kappa\in K$,  and $v\in V\setminus\{s_d,t_d\}$, from (\ref{proof:R3-R2:balance}) we have that
\begin{align*}
\sum_{a \in A^-(v)}\tilde{y}^{\kappa,a}_{\tau,d} = \sum_{a \in A^-(v)}\frac{\bar{y}^{a}_{\tau,s_d,t_d}}{|K|q_{s_d,t_d}} = \frac{\sum_{a \in A^-(v)}\bar{y}^{a}_{\tau,s_d,t_d}}{|K|q_{s_d,t_d}} = \frac{\sum_{a \in A^+(v)}\bar{y}^{a}_{\tau,s_d,t_d}}{|K|q_{s_d,t_d}} = \sum_{a \in A^+(v)}\tilde{y}^{\kappa,a}_{\tau,d},
\end{align*}
implying that (\ref{eqn:ip:y:v}) is satisfied.

For each $\tau\in \Pi$, $\kappa\in K$, $e\in E$, by (\ref{proof:R2-R2:sum})
 and (\ref{eqn:lpr3:def:w2}) we have that 
\begin{align*}
&\sum_{d\in D}(\tilde{y}^{\kappa,\alpha_1(e)}_{\tau,d} + \tilde{y}^{\kappa,\alpha_2(e)}_{\tau,d})
= \sum_{d\in D}(\frac{\bar{y}^{\alpha_1(e)}_{\tau,s_d,t_d}}{|K|q_{s_d,t_d}} + \frac{\bar{y}^{\alpha_2(e)}_{\tau,s_d,t_d}}{|K|q_{s_d,t_d}})\\
=&\sum_{s\in V}\sum_{t\in V\setminus \{s\}:q_{s,t}>0}(\frac{\bar{y}^{\alpha_1(e)}_{\tau,s,t}\cdot q_{s,t}}{|K|q_{s,t}} + \frac{\bar{y}^{\alpha_2(e)}_{\tau,s,t}\cdot q_{s,t}}{|K|q_{s,t}})\nonumber\\
=& \sum_{s\in V}\sum_{t\in V\setminus \{s\}}(\frac{\bar{y}^{\alpha_1(e)}_{\tau,s,t}}{|K|} + \frac{\bar{y}^{\alpha_2(e)}_{\tau,s,t}}{|K|})=
\sum_{s\in V}(\frac{\hat{y}^{\alpha_1(e)}_{\tau,s}}{|K|} + \frac{\hat{y}^{\alpha_2(e)}_{\tau,s}}{|K|})
\leq \frac{\hat{w}^{e}}{|K|} = \tilde{w}^{\kappa,e},
\end{align*}
implying that (\ref{eqn:ip:def:w2}) is satisfied.

Moreover, for each $\tau\in \Pi$, $d\in D$, and $\kappa\in K$, by (\ref{eqn:lpr3:yy}) we have that
\begin{align*}
\tilde{y}^{\kappa,\alpha_1(\tau)}_{\tau,d} + \tilde{y}^{\kappa,\alpha_2(\tau)}_{\tau,d}= \frac{\bar{y}^{\alpha_1(\tau)}_{\tau,s_d,t_d} + \bar{y}^{\alpha_2(\tau)}_{\tau,s_d,t_d}}{|K|q_{s_d,t_d}} = 0,
\end{align*}
implying that (\ref{eqn:ip:yy}) is satisfied.


Hence, $\tilde{y}$ and $\tilde{w}$ defined above form a feasible solution to model $\mathrm{LP_{R2}}$, implying that $z_{\mathrm{LP_{R3}}}\geq z_{\mathrm{LP_{R2}}}$. This, together with $z_{\mathrm{LP_{R3}}}\leq z_{\mathrm{LP_{R2}}}$, implies that $z_{\mathrm{LP_{R3}}}= z_{\mathrm{LP_{R2}}}$. Proposition~\ref{prop:R3-R2} is proved.
\end{proof}

\section{Solving LP Relaxation $\mathrm{LP_{R3}}$}
\label{sec:IP:relax:solution}
As shown in Section~\ref{sec:LP:relax}, $\mathrm{LP_{R3}}$ is as tight as $\mathrm{LP_{R1}}$ and $\mathrm{LP_{R2}}$, which is much tighter than $\mathrm{LP_{RWAP}}$. Compared with $\mathrm{LP_{R1}}$,  $\mathrm{LP_{R2}}$, and $\mathrm{LP_{RWAP-PPP}}$, $\mathrm{LP_{R3}}$ has significantly fewer variables and constraints. However, for large-sized instances of the RWAP-PPP, the size of $\mathrm{LP_{R3}}$ may still be too large to solve directly.

To solve the LP relaxation $\mathrm{LP_{R3}}$ of the RWAP-PPP, we develop a Benders decomposition algorithm in this section. First, for each $\tau\in \Pi$, given a vector $\bar{w}$ with each of its components $\bar{w}^e$ satisfying $0\leq \bar{w}^e\leq |K|$ for $e\in E$, we define a subproblem as follows, which is referred to as subproblem $F_{\tau}(\bar{w})$, and its optimal objective value is also denoted by $F_{\tau}(\bar{w})$:
\begin{align}
   & F_{\tau}(\bar{w}) = \min \epsilon \\
    \mbox{s.t.}~
& \sum_{a \in A^+(s)}y^{a}_{\tau,s} = \sum_{t\in T}q_{st}, \mbox{ $\forall s\in V$}, \label{eqn:lpr3:F:y:s}\\
& \sum_{a \in A^-(s)}y^{a}_{\tau,s} = 0, \mbox{ $\forall s\in V$}, \label{eqn:lpr3:F:y:s-}\\
& \sum_{a \in A^-(v)}y^{a}_{\tau,s} = \sum_{a \in A^+(v)}y^{a}_{\tau,s}+q_{sv}, \mbox{ $\forall s\in V,~v\in V\setminus\{s\}$}, \label{eqn:lpr3:F:y:v}\\
&  \sum_{s\in V}(y^{\alpha_1(e)}_{\tau,s} + y^{\alpha_2(e)}_{\tau,s}) \leq \bar{w}^{e}+\epsilon, \mbox{ $\forall e\in E$}, \label{eqn:lpr3:F:def:w2}\\
& y^{\alpha_1(\tau)}_{\tau,s}+y^{\alpha_2(\tau)}_{\tau,s}=0, \mbox{ $\forall s\in V$},\label{eqn:lpr3:F:yy}\\
& 0\leq y^{a}_{\tau,s}\leq \sum_{t\in T}q_{st}, \mbox{ $\forall
s\in V,~a\in A$}, \label{eqn:lpr3:F:y}\\
& \epsilon \geq 0.
\end{align}
Constraints (\ref{eqn:lpr3:F:def:w2}) \revxzb{are} relaxations of constraints (\ref{eqn:lpr3:def:w2}), and a new variable, $\epsilon$, is introduced to represent the maximum violation of constraints (\ref{eqn:lpr3:def:w2}). The objective of subproblem $F_{\tau}(\bar{w})$ is to minimize this maximum violation, $\epsilon$. It can be seen that subproblem $F_{\tau}(\bar{w})$ contains only $O(|V||E|)$ variables and $O(|V|^2)$ constraints, much fewer than $\mathrm{LP_{R3}}$.

We can now establish Proposition~\ref{prop:Ftauw} below, which utilizes $F_{\tau}(\bar{w})$ to present a condition on the feasibility of a solution. 
\begin{proposition}
\label{prop:Ftauw}
Given any $\bar{w}$ with $0\leq \bar{w}^e\leq |K|$ for $e\in E$, there exists $y$ such that $(\bar{w},y)$ forms a feasible solution to $\mathrm{LP_{R3}}$, if and only if $F_{\tau}(\bar{w})=0$ for all $\tau\in \Pi$.
\end{proposition}
\begin{proof}{Proof}
On the one hand, if $(\bar{w},y)$ forms a feasible solution to $\mathrm{LP_{R3}}$, then $y_{\tau,s}^{a}$ for $s\in V$ and $a\in A$ and $\epsilon=0$ form a feasible solution to the subproblem $F_{\tau}(\bar{w})$ for each $\tau\in \Pi$. This implies that $F_{\tau}(\bar{w})=0$ for all $\tau\in \Pi$.

On the other hand, for each $\tau\in \Pi$, if $F_{\tau}(\bar{w})=0$, then there exist $y_{\tau,s}^{a}$ for $s\in V$ and $a\in A$ that, together with $\epsilon=0$, form a feasible solution to the subproblem $F_{\tau}(\bar{w})$. Thus, it can be seen that $(\bar{w},y)$ forms a feasible solution to $\mathrm{LP_{R3}}$. Proposition~\ref{prop:Ftauw} is proved.
\end{proof}

We can see that each subproblem $F_{\tau}(\bar{w})$ defined above is an LP. Its dual problem can be formulated as follows, where link $\tau\in \Pi$, and $\beta$, $\phi$, $\gamma$, $\theta$, $\psi$, and $\zeta$ are dual variables associated with constraints (\ref{eqn:lpr3:F:y:s})--(\ref{eqn:lpr3:F:y}), respectively:
\begin{align*}
\max &\quad  \sum_{s\in V}\sum_{t\in T}q_{st}\beta_{s} + \sum_{s\in V}\sum_{v\in V\setminus \{s\}}q_{sv}\gamma_{sv} + \sum_{e\in E}\bar{w}^{e}\theta_{e} + \sum_{s\in V}\sum_{a\in A}\sum_{t\in T}q_{st}\zeta_{s,a}\\
\mbox{s.t.}~
& \gamma_{s,v}-\gamma_{s,u}+\theta_{e}+\zeta_{s,a} \leq 0, \mbox{ $\forall s\in V, e\in E\setminus \{\tau\}, a=( u,v) \in \{\alpha_1(e),\alpha_2(e)\}$ with $u\neq s, v\neq s$},\\
& \gamma_{s,v}-\gamma_{s,u}+\theta_{e}+\psi_s+\zeta_{s,a} \leq 0, \mbox{ $\forall s\in V, a=( u,v) \in \{\alpha_1(\tau),\alpha_2(\tau)\}$ with $u\neq s, v\neq s$},\\
& \beta_{s} +\gamma_{s,v}+\theta_{e}+\zeta_{s,a} \leq 0, \mbox{ $\forall s\in V, e\in E\setminus \{\tau\},a=( s,v) \in \{\alpha_1(e),\alpha_2(e)\}$ with $v\neq s$},\\
& \beta_{s} +\gamma_{s,v}+\theta_{e}+\psi_{s}+\zeta_{s,a} \leq 0, \mbox{ $\forall s\in V, a=( s,v) \in \{\alpha_1(\tau),\alpha_2(\tau)\}$ with $v\neq s$},\\
& \phi_s -\gamma_{v,s}+\theta_{e}+\zeta_{s,a} \leq 0, \mbox{ $\forall s\in V, e\in E\setminus \{\tau\},a=( v,s) \in \{\alpha_1(e),\alpha_2(e)\}$ with $v\neq s$},\\
& \phi_s-\gamma_{v,s}+\theta_{e}+\psi_{s}+\zeta_{s,a} \leq 0, \mbox{ $\forall s\in V, a=( v,s) \in \{\alpha_1(\tau),\alpha_2(\tau)\}$ with $v\neq s$},\\
& -\sum_{e\in E}\theta_{e}\leq 1, \\
& \theta_{e}\leq 0, \mbox{ $\forall e\in E$},\\
& \zeta_{s,a}\leq 0, \mbox{ $\forall s\in V,~a\in A$}.
\end{align*}
\normalsize
Let $\Gamma_{\tau}$ indicate the domain of the above dual problem of the subproblem $F_{\tau}(\bar{w})$. By strong duality, we know that $F_{\tau}(\bar{w})=0$ if and only if $ \sum_{s\in V}\sum_{t\in T}q_{st}\beta_{s} + \sum_{s\in V}\sum_{v\in V\setminus \{s\}}q_{sv}\gamma_{sv} + \sum_{e\in E}\bar{w}^{e}\theta_{e} + \sum_{s\in V}\sum_{a\in A}\sum_{t\in T}q_{st}\zeta_{s,a}= 0$ for all $(\beta,\phi,\gamma,\theta,\psi,\zeta)\in \Gamma_{\tau}$.

Let $\tau_0$ denote any link in $\Pi$. By Proposition~\ref{prop:Ftauw}, we can reformulate $\mathrm{LP_{R3}}$ as the following LP program, which is referred to as $\mathrm{LP'_{R3}}$:
\begin{align}
\mbox{($\mathrm{LP'_{R3}}$)}\quad & \min \sum_{e\in E}\bar{w}^{e} \label{eqn:lpr3prime:obj}\\
\mbox{s.t.}~ 
& \sum_{s\in V}\sum_{t\in T}q_{st}\beta_{s} + \sum_{s\in V}\sum_{v\in V\setminus \{s\}}q_{sv}\gamma_{sv} + \sum_{e\in E}\bar{w}^{e}\theta_{e} + \sum_{s\in V}\sum_{a\in A}\sum_{t\in T}q_{st}\zeta_{s,a} = 0, \nonumber\\
&
~~~~~~~~~~~~~~~~~~~~~~~~~~~~~~~~~~~~~~~~~~~~~~~~~~\mbox{ $\forall (\beta,\phi,\gamma,\theta,\psi,\zeta)\in \Gamma_{\tau},~\tau\in \Pi\setminus \{\tau_0\}$}, \label{eqn:lpr3prime:cut}\\
& \sum_{a \in A^+(s)}y^{a}_{\tau_0,s} = \sum_{t\in T}q_{st}, \mbox{ $\forall s\in V$}, \label{eqn:lpr3prime:F:y:s}\\
& \sum_{a \in A^-(v)}y^{a}_{\tau_0,s} = \sum_{a \in A^+(v)}y^{a}_{\tau_0,s}+q_{sv}, \mbox{ $\forall s\in V,~v\in V\setminus\{s\}$}, \label{eqn:lpr3prime:F:y:v}\\
&  \sum_{s\in V}(y^{\alpha_1(e)}_{\tau_0,s} + y^{\alpha_2(e)}_{\tau_0,s}) \leq \bar{w}^{e}, \mbox{ $\forall e\in E$}, \label{eqn:lpr3prime:F:def:w2}\\
& y^{\alpha_1(\tau_0)}_{\tau_0,s}+y^{\alpha_2(\tau_0)}_{\tau_0,s}=0, \mbox{ $\forall s\in V$},\label{eqn:lpr3prime:F:yy}\\
& 0\leq y^{a}_{\tau_0,s}\leq \sum_{t\in T}q_{st}, \mbox{ $\forall
s\in V,~a\in A$}, \label{eqn:lpr3prime:F:y}\\
& 0\leq \bar{w}^{e}\leq |K|, \mbox{ $\forall
e\in E$}. \label{eqn:lpr3prime:w}
\end{align}
\revxzc{In $\mathrm{LP'_{R3}}$, all the constraints of $\mathrm{LP_{R3}}$ defined for variables $y^{a}_{\tau,s}$ with $s\in V$ and $a\in A$ are replaced by constraints in (\ref{eqn:lpr3prime:cut}) for all the links $\tau\in \Pi\setminus \{\tau_0\}$. Here, $\tau_0$ is any link in $\Pi$. We keep those constrains defined for $y^{a}_{\tau_0,s}$ with $s\in V$ and $a\in A$ as valid inequalities, which, as we will elaborate later (after Proposition~\ref{prop:algorithm:speedup}), contribute to the acceleration of our solution algorithm for $\mathrm{LP'_{R3}}$.}

Let $z_{\mathrm{LP'_{R3}}}$ indicate the optimal objective value of $\mathrm{LP'_{R3}}$. Theorem~\ref{thm:LPR3} below shows that the new LP relaxation $\mathrm{LP'_{R3}}$ is equivalent to relaxation $\mathrm{LP_{R3}}$.

\begin{theorem}
\label{thm:LPR3}
$z_{\mathrm{LP'_{R3}}}=z_{\mathrm{LP_{R3}}}$.
\end{theorem}
\begin{proof}{Proof}
On the one hand, if $(\bar{w},y)$ forms a feasible solution to $\mathrm{LP_{R3}}$, then by Proposition~\ref{prop:R3-R2}, we have $F_{\tau}(\bar{w})=0$ for $\tau\in \Pi\setminus \{\tau_0\}$. By strong duality we obtain that constraints in (\ref{eqn:lpr3prime:cut}) are all satisfied by $w$ for all $\tau\in \Pi\setminus \{\tau_0\}$. Thus, it can be seen that $\bar{w}^{e}$ for $e\in E$ and $y_{\tau_0,s}^{a}$ for $s\in V$ and $a\in A$ form a feasible solution to $\mathrm{LP'_{R3}}$.

On the other hand, consider any feasible solution to $\mathrm{LP'_{R3}}$, consisting of $\bar{w}^{e}$ for $e\in E$ and $y_{\tau_0,s}^{a}$ for $s\in V$ and $a\in A$. Since the constraints in (\ref{eqn:lpr3prime:cut}) are all satisfied  by $w$ for all $\tau\in \Pi\setminus \{\tau_0\}$, by strong duality we obtain that $F_{\tau}(\bar{w})=0$ for $\tau\in \Pi\setminus \{\tau_0\}$. Due to constraints (\ref{eqn:lpr3prime:F:y:s})--(\ref{eqn:lpr3prime:w}), we know that $F_{\tau_0}(\bar{w})=0$.  By Proposition~\ref{prop:R3-R2}, we obtain that $(\bar{w},y)$ forms a feasible solution to $\mathrm{LP_{R3}}$.

Thus, since $\mathrm{LP_{R3}}$ and $\mathrm{LP'_{R3}}$ also have the same objective value, we obtain that $\mathrm{LP_{R3}}$ and $\mathrm{LP'_{R3}}$ are equivalent, implying that $z_{\mathrm{LP_{R3}}}=z_{\mathrm{LP'_{R3}}}$. Theorem~\ref{thm:LPR3} is proved.
\end{proof}

According to Theorem~\ref{thm:LPR3}, to solve relaxation $\mathrm{LP_{R3}}$, we only need to solve $\mathrm{LP'_{R3}}$. It can be seen that $\mathrm{LP'_{R3}}$ contains only $O(|V||E|)$ variables, fewer than $\mathrm{LP_{R3}}$. It can also be seen that when constraints $(\ref{eqn:lpr3prime:cut})$ are excluded, $\mathrm{LP'_{R3}}$ contains only $O(|V||E|)$ constraints, much fewer than $\mathrm{LP_{R3}}$.

In $\mathrm{LP'_{R3}}$, the number of constraints in (\ref{eqn:lpr3prime:cut}) is very large. We, therefore, add them iteratively. In particular, let $\Gamma'_{\tau}$ indicate the values of $(\beta,\phi,\gamma,\theta,\psi,\zeta)\in \Gamma_{\tau}$ whose associated constraints in (\ref{eqn:lpr3prime:cut}) have been added. We initialize $\Gamma'_{\tau}$ to be an empty set. We refer to the relaxation of $\mathrm{LP'_{R3}}$ with $\Gamma_{\tau}$ replaced by $\Gamma'_{\tau}$ as the \textit{restricted master problem}. In each iteration of our Benders decomposition algorithm, we first solve the restricted master problem to optimality, where the corresponding optimal solution is denoted by $(\tilde{w},\tilde{y}_{\tau_0})$. We then solve the subproblem $F_{\tau}(\tilde{w})$ for $\tau\in \Pi$. If $F_{\tau}(\tilde{w})= 0$ for all $\tau\in \Pi$, then we obtain that $(\tilde{w},\tilde{y}_{\tau_0})$ is an optimal solution to $\mathrm{LP'_{R3}}$. Otherwise, for each $\tau\in \Pi$, since $F_{\tau}(\tilde{w})\geq 0$, we have $F_{\tau}(\tilde{w})>0$. We then obtain the optimal solution $(\beta,\phi,\gamma,\theta,\psi,\zeta)\in \Gamma_{\tau}$ to the dual of the subproblem $F_{\tau}(\tilde{w})$, and add $(\beta,\phi,\gamma,\theta,\psi,\zeta)$ to $\Gamma'_{\tau}$ so as to include its associated constraint in (\ref{eqn:lpr3prime:cut}) in the restricted master problem.

To further speed up the algorithm described above, we establish Proposition~\ref{prop:algorithm:speedup} below.
\begin{proposition}
\label{prop:algorithm:speedup}
Given any $(\bar{w},y_{\tau_0})$ that satisfies constraints (\ref{eqn:lpr3prime:F:y:s})--(\ref{eqn:lpr3prime:w}) of $\mathrm{LP'_{R3}}$, we have $F_{\tau}(\bar{w})=0$ for each $\tau\in \Pi$ with $y_{\tau_0,s}^{\alpha_1(\tau)}=y_{\tau_0,s}^{\alpha_2(\tau)}=0$ for all $s\in V$.
\end{proposition}
\begin{proof}{Proof}
When $y_{\tau_0,s}^{\alpha_1(\tau)}=y_{\tau_0,s}^{\alpha_2(\tau)}=0$ for all $s\in V$, it can be verified that $y^{a}_{\tau_0,s}$ for $s\in V$ and $a\in A$ and $\epsilon=0$ form a feasible solution to $F_{\tau}(\bar{w})$. Thus, $F_{\tau}(\bar{w})=0$. Proposition~\ref{prop:algorithm:speedup} is proved.
\end{proof}

\begin{remark}[Acceleration]
\revxzc{Consider the optimal solution  $(\tilde{w},\tilde{y}_{\tau_0})$ obtained for the restricted master problem in each iteration of our algorithm described above. Let $\Pi'$ indicate the set of links $\tau\in \Pi$ with $\tilde{y}_{\tau_0,s}^{\alpha_1(\tau)}=\tilde{y}_{\tau_0,s}^{\alpha_2(\tau)}=0$ for all $s\in V$. Proposition~\ref{prop:algorithm:speedup} implies that $F_{\tau}(\bar{w})=0$ for every link $\tau\in \Pi'$. As a result, we only need to solve the subproblems $F_{\tau}(\tilde{w})$ for $\tau\in \Pi\setminus \Pi'$. This reduction in the number of subproblems to solve significantly accelerates the Benders decomposition algorithm. It also demonstrates the benefit of having model $\mathrm{LP'_{R3}}$ keep the constraints defined for $y^{a}_{\tau_0,s}$ with $s\in V$ and $a\in A$ as valid inequalities.}
\end{remark}

\revxzc{Algorithm~\ref{alg:benders} below summarizes our Benders decomposition algorithm to solve $\mathrm{LP_{R3}}$. In its Step~5, when $F_{\tau}(\tilde{w})=0$ for all $\tau\in \Pi\setminus  \Pi'$, we know that $\tilde{w}$ and $\tilde{y}_{\tau_0}$ form an optimal solution to $\mathrm{LP'_{R3}}$, which, according to the proof of Theorem~\ref{thm:LPR3}, is also  an optimal solution to $\mathrm{LP_{R3}}$.}

\begin{algorithm}[h]
\caption{Benders Decomposition Algorithm for Solving $\mathrm{LP_{R3}}$}
\label{alg:benders}
\begin{enumerate}[itemsep=0pt,parsep=0pt]
\item Initialize $\Gamma'_{\tau}=\emptyset$ for each $\tau\in \Pi$.
\item Solve the restricted master problem of $\mathrm{LP'_{R3}}$ to obtain its optimal solution $(\tilde{w},\tilde{y}_{\tau_0})$, where $\tau_0$ is chosen arbitrarily from $\Pi$.
\item Let $\Pi'$ indicate the set of links $\tau\in \Pi$ with $\tilde{y}_{\tau_0,s}^{\alpha_1(\tau)}=\tilde{y}_{\tau_0,s}^{\alpha_2(\tau)}=0$ for all $s\in V$.
\item Solve subproblem $F_{\tau}(\tilde{w})$ for each $\tau\in \Pi\setminus \Pi'$.
\item If $F_{\tau}(\tilde{w})=0$ for all $\tau\in \Pi\setminus  \Pi'$, then 
return \revxzc{$(\tilde{w},\tilde{y}_{\tau_ 0})$ } as the optimal solution to $\mathrm{LP'_{R3}}$.
\item Otherwise, for each $\tau\in \Pi\setminus \Pi'$ with $F_{\tau}(\tilde{w})>0$, add to $\Gamma'_{\tau}$ the optimal solution $(\beta,\phi,\gamma,\theta,\psi,\zeta)\in \Gamma_{\tau}$ to the dual of $F_{\tau}(\tilde{w})$, and go to Step~2.
\end{enumerate}
\end{algorithm}

\section{Computational Experiment}
\label{sec:experiments}
Our computational experiment aims to examine (i) the performance of the newly proposed LP relaxation $\mathrm{LP_{R3}}$ in producing lower bounds for the RWAP-PPP, and (ii) the performance of the newly developed Benders decomposition algorithm (Algorithm~\ref{alg:benders}) in solving $\mathrm{LP_{R3}}$. Our Benders decomposition algorithm was coded in C++ and run on a PC in Windows 11 system with an Intel(R) Core(TM) i9-11950H 2.60GHz CPU (8 Cores) and 32 GB of RAM. An open-source LP solver of CLP 1.17.7 \citep{forrest2022coinclp} was used to solve all the restricted master problems and the subproblems, encountered in our algorithm. To establish benchmark results, we also employed a widely recognized commercial LP solver of CPLEX 22.1.0 to solve $\mathrm{LP_{RWAP}}$ and $\mathrm{LP_{R3}}$, for which the time limit was set to be 3 hours. Instead of directly solving $\mathrm{LP_{RWAP}}$, we initially transformed it into an equivalent aggregated reformulation using the approach proposed by \cite{jaumard2006ilp}. This allowed us to reduce the number of decision variables and constraints involved in $\mathrm{LP_{RWAP}}$. \revxzc{(It is worthy noting that we also conducted tests by adopting the CPLEX solver for the Benders decomposition algorithm and the CLP solver for directly solving $\mathrm{LP_{RWAP}}$ and $\mathrm{LP_{R3}}$. However, the resulting performances were inferior compared to the adoption introduced above.)}



Our computational experiment was based on 16 practical problem instances from Huawei Technologies. Due to confidentiality, original problem instances and their computational results cannot be disclosed. Therefore, we only present results for instances that were generated by making some random modifications to the original problem instances. Specifically, the vertices were shuffled randomly, and the vertices, edges, and requests were randomly modified while ensuring the feasibility of solutions. Nevertheless, we have observed that these modified test instances demonstrate similar statistical characteristics, and their computational results yield findings that align with those obtained from the original problem instances. The statistical characteristics of our randomized test instances are illustrated in columns $|V|$, $|E|$, and $|D|$ of Table~\ref{tab:result}. The number of available wavelengths $|K|$ equals 80. All links in $E$ are possible to fail so that $\Pi=E$. These instances will be made available online for future studies. 

\newcommand{\comments}[1]{%
    \linespread{0.0}\vspace{-10pt}%
    \raggedright\small\textit{#1}%
}

\begin{table}[t]
{
     \setlength{\tabcolsep}{1.5mm}
    \centering
    \caption{Computational Results on Solving $\mathrm{LP_{RWAP}}$ and $\mathrm{LP_{R3}}$.}
    \label{tab:result}
{\small    
\begin{tabular}{lllllllllllllll}			\hline																				
\multirow{2}{*}{No.}	&	\multirow{2}{*}{$|V|$}	&	\multirow{2}{*}{$|E|$}	&	\multirow{2}{*}{$|D|$}	&	\multirow{2}{*}{UB}	&&	\multicolumn{3}{c}{$\mathrm{LP}_{\mathrm{RWAP}}$}			&&	\multicolumn{5}{c}{$\mathrm{LP_{R3}}$}									\\ \cline{7-9} \cline{11-15}
	&		&		&		&		&&	Obj	& Gap\%	&	$\mathrm{T_{CPLEX}}$	&&	Obj	&	Im\%	&	Gap\%	&	$\mathrm{T_{Benders}}$	&	$\mathrm{T_{CPLEX}}$	\\\hline
1	&	59	&	184	&	591	&	2019	&&	1626 & 24.2	&	0.1	&&	1887	&	16.1	&	7.0	&	53.6 	&	40.0	\\
2	&	60	&	152	&	443	&	1675	&&	1239 & 35.2	&	0.1	&&	1532	&	23.6	&	9.3	&	6.0 	&	6.4	\\
3	&	70	&	227	&	689	&	2788	&&	2265 & 23.1	&	0.3	&&	2606	&	15.1	&	7.0	&	221.0 	&	195.7	\\
4	&	78	&	155	&	690	&	2257	&&	1404 & 60.8	&	0.1	&&	2142	&	52.6	&	5.4	&	13.2 	&	66.9	\\
5	&	80	&	256	&	784	&	3168	&&	2588 & 22.4	&	0.3	&&	2956	&	14.2	&	7.2	&	275.3 	&	357.1	\\
6	&	89	&	185	&	765	&	3551	&&	2241 & 58.5	&	0.1	&&	3191	&	42.4	&	11.3	&	30.7 &	32.6	\\
7	&	90	&	288	&	906	&	3873	&&	3238 & 19.6	&	0.5	&&	3623	&	11.9	&	6.9	&	697.0 	&	404.8	\\
8	&	100	&	330	&	983	&	5834	&&	4640 & 25.7	&	0.5	&&	5461	&	17.7	&	6.8	&	883.2 	&	4724.9	\\
9	&	119	&	249	&	899	&	4572	&&	3030 & 50.9	&	0.4	&&	4116	&	35.8	&	11.1	&	112.4 	&	138.0	\\
10	&	128	&	421	&	1852	&	10349	&&	8562 & 20.9	&	2.0	&&	9712	&	13.4	&	6.6	&	6028.0 	&	NA (0.0\%)	\\
11	&	131	&	441	&	2699	&	13109	&&	11244 & 16.6	&	1.9	&&	12462	&	10.8	&	5.2	&	7961.5 	&	NA (0.0\%)	\\
12	&	152	&	359	&	1801	&	9007	&&	6173 & 45.9	&	0.5	&&	7892	&	27.8	&	14.1	&	997.9 &	7581.7	\\
13	&	173	&	386	&	579	&	4465	&&	3130 & 42.7	&	0.7	&&	4084	&	30.5	&	9.3	&	796.9 	&	7408.8	\\
14	&	178	&	394	&	1579	&	7440	&&	5451 & 36.5	&	1.9	&&	6716	&	23.2	&	10.8	&	8939.6 	&	NA	\\
15	&	225	&	441	&	662	&	6205	&&	4255 & 45.8	&	1.2	&&	5723	&	34.5	&	8.4	&	949.4 	&	NA	\\
16	&	292	&	588	&	1870	&	5820	&&	3667 & 58.7	&	1.3	&&	5246	&	43.1	&	10.9	&	8402.5 	&	NA	\\\hline
\end{tabular}																				
}
}

\textit {Note: If $\mathrm{LP_{R3}}$ is not solved to optimum by CPLEX within the 3-hour time limit, ``NA" is displayed in $\mathrm{T_{CPLEX}}$, where the number in the bracket indicates the percentage gap between the best lower bound found within the time limit and the optimum objective value of $\mathrm{LP_{R3}}$.}
\end{table}

To examine the performance of $\mathrm{LP_{R3}}$, we compare the lower bound produced by $\mathrm{LP_{R3}}$ with that by $\mathrm{LP_{RWAP}}$, as shown in columns Obj of Table~\ref{tab:result}. The improvement percentage, defined as $(z_{\mathrm{LP_{R3}}}-z_{\mathrm{LP_{RWAP}}})/z_{\mathrm{LP_{RWAP}}}\times 100\%$, is presented in column Im\%. The higher its value is, the more improvement is made in the lower bound by $\mathrm{LP_{R3}}$ against $\mathrm{LP_{RWAP}}$. Moreover, we also compare the lower bound produced by $\mathrm{LP_{R3}}$ with the best-known upper bound produced by a heuristic algorithm currently used in Huawei Technologies, which is shown in column UB. The optimality gap, defined as $(\mathrm{UB}-z_{\mathrm{LP_{R3}}})/z_{\mathrm{LP_{R3}}}\times 100\%$, is presented in column Gap\% \revxzc{under $\mathrm{LP_{R3}}$} of Table~\ref{tab:result}. The lower its value, the tighter the lower bound produced by $\mathrm{LP_{R3}}$. For the purpose of comparisons, column Gap\%\revxzc{, defined as $(\mathrm{UB}-z_{\mathrm{LP_{RWAP}}})/z_{\mathrm{LP_{RWAP}}}\times 100\%$, is also presented for $\mathrm{LP_{RWAP}}$}.


From column Im\% under $\mathrm{LP_{R3}}$ of Table~\ref{tab:result}, we can see that our new LP relaxation $\mathrm{LP_{R3}}$ produces significantly better lower bounds for the RWAP-PPP than $\mathrm{LP_{RWAP}}$. Among all the 16 instances, the improvement ratio ranges from 10.8\% to 52.6\%, with an average of 25.8\%, which is significant. It is worth noting that Propositions~\ref{proposition:zlpr2} and \ref{prop:R3-R2}  highlight the potential for $\mathrm{LP_{R3}}$ to achieve significantly improved lower bounds over $\mathrm{LP_{RWAP}}$ in certain extreme situations, although such situations may not commonly occur in practice. Empirical evidence from our experimental results here supports the observation that the improvement is indeed substantial. 
\revxzc{Moreover, from columns Gap\% of Table~\ref{tab:result}, we can also observe that our new LP relaxation $\mathrm{LP_{R3}}$ achieves an optimality gap ranging from 5.2\% to 14.1\%, with an average of only 8.6\%. In contrast, the LP relaxation $\mathrm{LP_{RWAP}}$ has a significantly larger optimality gap ranging from 16.6\% to 60.8\%, with an average of 36.7\%.}

To examine the performance of our Benders decomposition algorithm, we compare its computational time with that of the CPLEX LP solver in solving $\mathrm{LP_{R3}}$. Their running times (in seconds) for solving $\mathrm{LP_{R3}}$ are shown in columns $\mathrm{T_{Benders}}$ and $\mathrm{T_{CPLEX}}$ under $\mathrm{LP_{R3}}$ of Table~\ref{tab:result}. When the CPLEX LP solver was applied to solve $\mathrm{LP_{R3}}$ directly, it was unable to solve 5 of the 16 instances to optimality within the 3 hour time limit. Such instances are indicated by ``NA" in column $\mathrm{T_{CPLEX}}$ under $\mathrm{LP_{R3}}$ of Table~\ref{tab:result}. Out of such 5 instances, instances 10 and 11 resulted in the CPLEX LP solver finding the best lower bound that equal the optimal objective value of $\mathrm{LP_{R3}}$, but being unable to prove their optimality within the time limit. Instances 14, 15, and 16 resulted in CPLEX LP solver failing to find any positive lower bound on the optimal objective value of $\mathrm{LP_{R3}}$. 


From columns $\mathrm{T_{Benders}}$ and $\mathrm{T_{CPLEX}}$ under $\mathrm{LP_{R3}}$ of Table~\ref{tab:result}, we can see that our Benders decomposition algorithm is significantly more efficient than the CPLEX LP solver in solving the new LP relaxation $\mathrm{LP_{R3}}$. The Benders decomposition algorithm terminated and solved $\mathrm{LP_{R3}}$ to optimality within 3 hours for all the 16 instances, while the CPLEX LP solver did so for only 11 instances. For these 11 instances, which were solved to optimality by both algorithms, the Benders decomposition algorithm took only 371.6 seconds on average, significantly faster than the CPLEX LP solver, which took 1905.2 seconds on average.

The computational results have demonstrated that the newly proposed LP relaxation $\mathrm{LP_{R3}}$ and the newly developed Benders decomposition algorithm can be applied to effectively assess solution qualities for the RWAP-PPP. However, although our new LP relaxation $\mathrm{LP_{R3}}$ provides significant better lower bounds for the RWAP-PPP than $\mathrm{LP_{RWAP}}$, it is considerably more difficulty to solve. From column  $\mathrm{T_{CPLEX}}$ under $\mathrm{LP_{RWAP}}$ of Table~\ref{tab:result}, we can see that the CPLEX LP solvers have demonstrated remarkable efficiency in solving $\mathrm{LP_{RWAP}}$, with an average running time of only 0.7 seconds. This is attributed to the reduction of variables and constraints achieved through the equivalent transformation of $\mathrm{LP_{RWAP}}$ as proposed in \cite{jaumard2006ilp}.

\section{Conclusions}
\label{sec:conclusions}
\revxzc{We study LP relaxations of the RWAP-PPP, aiming to derive an effective lower bound on the optimal objective value of the RWAP-PPP that can be efficiently computed for large-scale practical telecommunication networks. To achieve this, we propose a novel LP relaxation of the RWAP-PPP by relaxing some of the linking constraints of a direct LP relaxation of the RWAP-PPP, so that the numbers of variables and constraints are significantly reduced. The new LP relaxation yields improved lower bounds for the RWAP-PPP compared to a direct LP relaxation of the RWAP. We prove that these improvements can be arbitrarily large. By incorporating a set of valid inequalities, we then develop a Benders decomposition algorithm. Computational results show that the algorithm can efficiently solve our new LP relaxation of the RWAP-PPP for some large-scale practical networks. The new lower bounds obtained are found to significantly improve the lower bounds obtained from a direct LP relaxation of the RWAP by 10.8\% to 52.6\%. As a result, our new LP relaxation and its solution algorithm can be used to effectively assess the quality of heuristic solutions for the RWAP-PPP, which has great research and practical values.}

In the future, we will investigate the improvement of relaxations of the RWAP-PPP, as well as the algorithms required to solve the relaxations. Deriving near-optimal heuristic solutions for the RWAP-PPP using these relaxations is also interesting, but it is challenging, particularly for large-scale networks.
Furthermore, the optimization model and solution method developed in our study can be adapted to other problem settings of the RWAP-PPP. This includes scenarios with different objective functions commonly considered in the RWAP literature. For instance, if the objective is to minimize the total number of used wavelengths (as seen in \cite{lee2002optimization} and \cite{hu2004traffic}), we can introduce new binary decision variables $w^{\kappa}$ to indicate whether wavelength $\kappa$ is used or not for each $\kappa\in K$. By replacing $w^{\kappa,e}$ with $w^{\kappa}$, the IP model $\mathrm{IP_{RWAP-PPP}}$ can be adapted for the new objective. Consequently, the LP relaxations and the Benders decomposition algorithm can also be adjusted accordingly. Similarly, if the objective is to minimize the maximum number of wavelengths assigned over all the links (as discussed in \cite{jaumard2006ilp}), we can introduce a new decision variable $\theta$ to represent the maximum number of wavelengths assigned. By replacing the objective with $\min \theta$ and incorporating new constraints $\theta\geq \sum_{k\in K}w^{k,e}$ for all $e\in E$, the IP model $\mathrm{IP_{RWAP-PPP}}$ can be adapted for this objective. Again, the LP relaxations and the Benders decomposition algorithm can also be modified accordingly. In future studies, for these new problem settings, it would be interesting to investigate the tightness of the lower bounds obtained from the LP relaxations and assess the efficiency of the Benders decomposition algorithm.	

\bibliographystyle{apalike}
\bibliography{references}
\setcounter{equation}{0}
\renewcommand{\theequation}{A.\arabic{equation}}
\allowdisplaybreaks

\begin{appendices}
\setcounter{equation}{0}
\renewcommand{\theequation}{\thesection.\arabic{equation}}
\section{An Illustartive Example for the RWAP-PPP}
\label{sec:app:eg:rwappp}

Consider an instance of the RWAP-PPP shown in Figure~\ref{fig:eg} below. 

\begin{figure}[h]
    \centering
    \caption{Example of the RWAP-PPP with two communication requests $(1,3)$ and $(4,3)$.}
    \label{fig:eg}
    \scalebox{0.8}{	\includegraphics{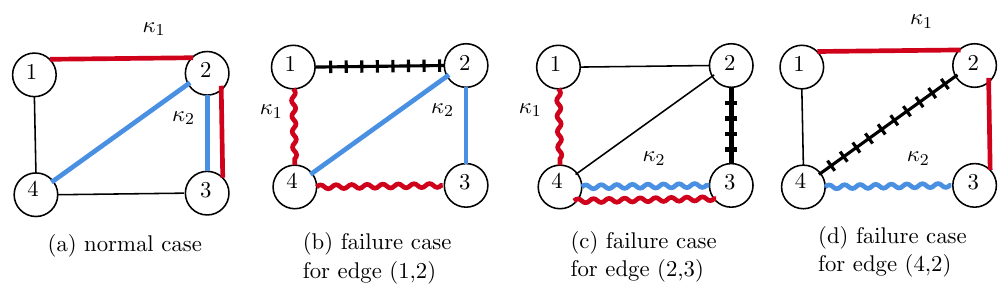}}
\end{figure}

For request $(1,3)$, we assign it a working path denoted by $(1,2,3)$ and assign each link of the path, including $\{1,2\}$ and $\{2,3\}$, the same working wavelength denoted by $\kappa_1$. For request $(4,3)$, we assign it a working path denoted by $(4,2,3)$ and assign each link of the path, including $\{4,2\}$ and $\{2,3\}$, the same working wavelength denoted by $\kappa_2$, where $\kappa_2$ and $\kappa_1$ are different. As shown in Figure~\ref{fig:eg}(a), these form a feasible solution to the RWAP with the total number of assigned wavelengths of all the links equal to 4, including one wavelength for link $\{1,2\}$, one wavelength for link $\{4,2\}$, and two wavelengths for link $\{2,3\}$. Next, consider $\Pi = \{\{1,2\},\{2,3\},\{2,4\}\}$, indicating that each of these three links in $\Pi$ is possible to fail. We assign backup paths and backup wavelengths for the following three possible cases of link failures:
\begin{itemize}
    \item For the case where link $\{1,2\}$ has failed, as shown in Figure~\ref{fig:eg}(b), request $(4,3)$ can still be satisfied by its working path $(4,2,3)$ and working wavelength $\kappa_2$. To satisfy request $(1,3)$ we assign it a backup path denoted by $(1, 4, 3)$ and assign $\kappa_1$ as the backup wavelength to all the links of the backup path, including $\{1,4\}$ and $\{4,3\}$. 
    \item For the case where link $\{2,3\}$ has failed, as shown in Figure~\ref{fig:eg}(c), to satisfy request $(1,3)$ we assign it a backup path denoted by $(1,4,3)$ and assign $\kappa_1$ as the backup wavelength to all the links of the path, including $\{1,4\}$ and $\{4,3\}$. To satisfy request $(4,3)$, we assign it a backup path denoted by $(4,3)$ and assign $\kappa_2$ as the backup wavelength to the link $\{4,3\}$ of the backup path. 
    \item For the case where link $\{4,2\}$ has failed, as shown in Figure~\ref{fig:eg}(d), request $(1,3)$ can still be satisfied by its working path $(1,2,3)$ and working wavelength $\kappa_1$. To satisfy request $(4,3)$ we assign it a backup path denoted by $(4,3)$ and assign $\kappa_2$ as the backup wavelength to the only link $\{4,3\}$ of the backup path.
\end{itemize}
Both the working and backup paths, as well as their assigned working and backup wavelengths, form a feasible solution to the RWAP-PPP, in which wavelength $\kappa_1$ is assigned to links $\{1,2\}$, $\{2,3\}$, $\{1,4\}$, $\{4,3\}$, and wavelength $\kappa_2$ is assigned to links $\{4,2\}$, $\{2,3\}$, and $\{4,3\}$. Hence, the objective value of this solution to the RWAP-PPP equals 7.
\end{appendices}

\end{document}